\documentclass[12pt]{article}%
\usepackage{amsmath}
\usepackage{amssymb}
\usepackage{epsfig}
\usepackage{amscd}
\usepackage{amsfonts}
\usepackage{graphicx}%
\providecommand{\U}[1]{\protect\rule{.1in}{.1in}}
\newtheorem{theorem}{Theorem}

\newtheorem{corollary}[theorem]{Corollary}

\newtheorem{example}[theorem]{Example}

\newtheorem{lemma}[theorem]{Lemma}

\newtheorem{proposition}[theorem]{Proposition}
\newtheorem{remark}[theorem]{Remark}

\numberwithin{table}{section}
\newenvironment{proof}[1][Proof]{\noindent\textbf{#1.} }{\ \rule{0.5em}{0.5em}}
\begin{document}

\title{Uniform approximation of the Cox-Ingersoll-Ross process}
\author{G.N. Milstein\thanks{Ural Federal University, Lenin Str.~51, 620083
Ekaterinburg, Russia; email: Grigori.Milstein@usu.ru}
\and J.G.M. Schoenmakers\thanks{Weierstrass-Institut f\"{u}r Angewandte Analysis
und Stochastik, Mohrenstrasse 39, 10117 Berlin, Germany; email:
schoenma@wias-berlin.de}}
\maketitle

\begin{abstract}
The Doss-Sussmann (DS) approach is used for uniform simulation of the
Cox-Ingersoll-Ross (CIR) process. The DS formalism allows to express
trajectories of the CIR process through solutions of some ordinary
differential equation (ODE) depending on realizations of a Wiener process
involved. By simulating the first-passage times of the increments of the
Wiener process to the boundary of an interval and solving the ODE, we
uniformly approximate the trajectories of the CIR process. In this respect
special attention is payed to simulation of trajectories near zero. From a
conceptual point of view the proposed method gives a better quality of
approximation (from a path-wise point of view) than standard, or even exact
simulation of the SDE at some discrete time grid.



\end{abstract}

\section{Introduction}

The Cox-Ingersoll-Ross process $V(t)$ is determined by the following
stochastic differential equation (SDE)%
\begin{equation}
dV(t)=k(\lambda-V(t))dt+\sigma\sqrt{V}dw(t),\ V(t_{0})=V_{0}, \label{In1}%
\end{equation}
where $k,\ \lambda,\ \sigma$ are positive constants, and $w$ is a scalar
Brownian motion. Due to \cite{CIR} this process has become very popular in
financial mathematical applications. The CIR process is used in particular as
volatility process in the Heston model \cite{Hest}. It is known (\cite{IW},
\cite{KS}) that for $V_{0}>0$ there exists a unique strong solution
$V_{t_{0},V_{0}}(t)$ of (\ref{In1}) for all $t\geq t_{0}\geq0$. The
CIR\ process $V(t)=V_{t_{0},V_{0}}(t)$ is positive in the case $2k\lambda
\geq\sigma^{2}$ and nonnegative in the case $2k\lambda<\sigma^{2}.$ Moreover,
in the last case the origin is a reflecting boundary.

As a matter of fact, (\ref{In1}) does not satisfy the global Lipschitz
assumption. The difficulties arising in a simulation method for (\ref{In1})
are connected with this fact and with the natural requirement of preserving
nonnegative approximations. A lot of approximation methods for the CIR
processes are proposed. For an extensive list of articles on this subject we
refer to \cite{Anders} and \cite{DNS}. Besides \cite{Anders} and \cite{DNS} we
also refer to \cite{A1, A2, HM, HMS}, where a number of discretization schemes
for the CIR process can be found. Further we note that in \cite{MT2} a weakly
convergent fully implicit method is implemented for the Heston model. Exact
simulation of (\ref{In1}) is considered in \cite{BK, Gl} (see \cite{Anders} as well).

In this paper, we consider uniform pathwise approximation of $V(t)$ on an
interval $[t_{0},t_{0}+T]$ using the Doss-Sussmann transformation
(\cite{Doss}, \cite{SS}, \cite{RW}) which allows for expressing any trajectory
of $V(t)$ by the solution of some ordinary differential equation that depends
on the realization of $w(t).$ The approximation $\overline{V}(t)$ will be
uniform in the sense that the path-wise error will be uniformly bounded, i.e.%
\begin{equation}
\sup_{t_{0}\leq t\leq t_{0}+T}\left\vert \overline{V}(t)-V(t)\right\vert \leq
r\text{ \ \ almost surely,} \label{Uni}%
\end{equation}
where $r>0$ is fixed in advance. In fact, by simulating the first-passage
times of the increments of the Wiener process to the boundary of an interval
and solving this ODE, we approximately construct a generic trajectory of
$V(t).$ Such kind of simulation is more simple than the one proposed in
\cite{BK} and moreover has the advantage of uniform nature. Let us consider
the simulation of a standard Brownian motion $W$ on a fixed time grid
\[
t_{0},t_{i},...,t_{n}=T.
\]
Although $W$ may be even exactly simulated at the grid points, the usual
piecewise linear interpolation%
\[
\overline{W}(t)=\frac{t_{i+1}-t}{t_{i+1}-t_{i}}W(t_{i})+\frac{t-t_{i}}%
{t_{i+1}-t_{i}}W(t_{i+1})
\]
is not uniform in the sense of (\ref{Uni}). Put differently, for any (large)
positive number $A,$ there is always a positive probability (though possibly
small) that%
\[
\sup_{t_{0}\leq t\leq t_{0}+T}\left\vert \overline{W}(t)-W(t)\right\vert >A.
\]
Therefore, for path dependent applications for instance, such a standard, even
exact, simulation method may be not desirable and a uniform method preserving
(\ref{Uni}) may be preferred.

We note that the original DS results rely on a global Lipschitz assumption
that is not fulfilled for (\ref{In1}). We therefore have introduced the DS
formalism that yields a corresponding ODE which solutions are defined on
random time intervals. If $V$ gets close to zero however, the ODE becomes
intractable for numerical integration and so, for the parts of a trajectory
$V(t),$ that are close to zero, we are forced to use some other (not DS)
approach. For such parts we here propose a different uniform simulation
method. Another restriction is connected with the condition $\alpha:=\left(
4k\lambda-\sigma^{2}\right)  /8>0.$ We note that the case $\alpha>0$ is more
general than the case $2k\lambda\geq\sigma^{2}$ that ensures positivity of
$V(t),$ and stress that in the literature virtually all convergence proofs for
methods for numerical integration of (\ref{In1}) are based on the assumption
$2k\lambda\geq\sigma^{2}.$ We expect that the results here obtained for
$\alpha>0$ can be extended to the case where $\alpha\leq0,$ however in a
highly nontrivial way. Therefore, the case $\alpha\leq0$ will be considered in
a subsequent work.

The next two sections are devoted to DS formalism in connection with
(\ref{In1}) and to some auxiliary propositions. In Sections~4 and 5 we deal
with the one-step approximation and the convergence of the proposed method,
respectively. Section~6 is dedicated to the uniform construction of
trajectories close to zero.

\section{The Doss-Sussmann transformation}

\textbf{2.1} Due to the Doss-Sussmann approach (\cite{Doss}, \cite{IW},
\cite{RW}, \cite{SS}), the solution of (\ref{In1}) may be expressed in the
form%
\begin{equation}
V(t)=F(X(t),w(t)), \label{DS1}%
\end{equation}
where $F=F(x,y)$ is some deterministic function and $X(t)$ is the solution
\ of some ordinary differential equation depending on the part $w(s),\ 0\leq
s\leq t,$ of the realization $w(\cdot)$ of the Wiener process $w(t).$

Let us recall the Doss-Sussmann formalism according to \cite{RW}, V.28. In
\cite{RW} one consideres the Stratonovich SDE
\begin{equation}
dV(t)=b(V)dt+\gamma(V)\circ dw(t). \label{DS5}%
\end{equation}
The function $F=F(x,y)$ is found from the equation
\begin{equation}
\frac{\partial F}{\partial y}=\gamma(F),\ F(x,0)=x, \label{DS6}%
\end{equation}
and $X(t)$ is found from the ODE%
\begin{equation}
\frac{dX}{dt}=\frac{1}{\partial F/\partial x(X(t),w(t))}%
b(F(X(t),w(t)),\ X(0)=V(0). \label{DS7}%
\end{equation}

It turns out that application of the DS formalism after the Lamperti
transformation $U(t)=\sqrt{V(t)}$ (see \cite{DNS}) leads to more simple
equations. The Lamperti transformation yields the following SDE with additive
noise%
\begin{align}
dU  &  =(\frac{\alpha}{U}-\frac{k}{2}U)dt+\frac{\sigma}{2}dw,\ U(0)=\sqrt
{V(0)}>0,\text{ \ \ where}\label{DS8}\\
\alpha &  =\dfrac{4k\lambda-\sigma^{2}}{8}. \label{DS8a}%
\end{align}
Let us seek the solution of (\ref{DS8}) in the form
\begin{equation}
U(t)=G(Y(t),w(t)) \label{DS2}%
\end{equation}
in accordance with (\ref{DS1})-(\ref{DS7}). Because the Ito and Stratonovich
forrms of equation (\ref{DS8}) coincide, we have%
\[
b(U)=\frac{\alpha}{U}-\frac{k}{2}U,\ \gamma(U)=\frac{\sigma}{2}.
\]
The function $G=G(y,z)$ is found from the equation
\[
\frac{\partial G}{\partial z}=\frac{\sigma}{2},\ G(y,0)=y,
\]
i.e.,%
\begin{equation}
G(y,z)=y+\frac{\sigma}{2}z, \label{DS3}%
\end{equation}
and $Y(t)$ is found from the ODE%
\begin{equation}
\frac{dY}{dt}=\frac{\alpha}{Y+\frac{\sigma}{2}w(t)}-\frac{k}{2}(Y+\frac
{\sigma}{2}w(t)),\ Y(0)=U(0)=\sqrt{V(0)}>0. \label{DS4}%
\end{equation}
From (\ref{DS2}), (\ref{DS3}), and solution of (\ref{DS4}), we formally obtain
the solution $U(t)$ of (\ref{DS8}):%
\begin{equation}
U(t)=Y(t)+\frac{\sigma}{2}w(t). \label{DS9}%
\end{equation}
Hence%
\begin{equation}
V(t)=U^{2}(t)=(Y(t)+\frac{\sigma}{2}w(t))^{2}. \label{DS10}%
\end{equation}
\textbf{2.2} Since the Doss-Sussmann results rely on a global Lipschitz
assumption that is not fulfilled for (\ref{In1}), solution (\ref{DS10}) has to
be considered only formally. In this section we therefore give a direct proof
of the following more precise result.

\begin{proposition}
\label{Proposition 1} {Let }$Y(0)=U(0)=\sqrt{V(0)}>0.$ {Let }$\tau$ {be the
following stopping time: }%
\[
\tau:=\inf\{t:V(t)=0\}.
\]
{Then equation (\ref{DS4}) has a unique solution }$Y(t)$ {on the interval
}$[0,\tau),$ {the solution} $U(t)$ {of (\ref{DS8}) is expressed by formula
(\ref{DS9}) on this interval, and} $V(t)$ {is expressed by (\ref{DS10}).}
\end{proposition}

\begin{proof}
Let $(w(t),\ V(t))$ be the solution of the SDE system
\[
dw=dw(t),\ dV=k(\lambda-V)dt+\sigma\sqrt{V(t)}dw\left(  t\right)  ,
\]
which satisfies the initial conditions $w(0)=0,\ V(0)>0.$ Then $U(t)=\sqrt
{V(t)}>0$ is a solution of (\ref{DS8}) on the interval $[0,\tau).$ Consider
the function $Y(t)=U(t)-\frac{\sigma}{2}w(t),\ 0\leq t<\tau.$ Clearly,
$Y(t)+\frac{\sigma}{2}w(t)>0$ on $[0,\tau).$ Due to Ito's formula, we get%
\[
dY(t)=dU(t)-\frac{\sigma}{2}dw(t)=\frac{\alpha dt}{Y+\frac{\sigma}{2}%
w(t)}-\frac{k}{2}(Y+\frac{\sigma}{2}w(t))dt,
\]
i.e., the function $U(t)-\frac{\sigma}{2}w(t)$ is a solution of (\ref{DS4}).
The uniqueness of $Y(t)$ follows from the uniqueness of $V(t).$
\end{proof}

\medskip

\noindent\textbf{2.3} So far we were starting at the moment $t=0$. It is
useful to consider the Doss-Sussmann transformation with an arbitrary initial
time $t_{0}>0$ (which even may be a stopping time, for example, $0\leq
t_{0}<\tau$). In this case, we obtain instead of (\ref{DS4}) for%
\[
Y=Y(t;t_{0})=U(t)-\frac{\sigma}{2}(w(t)-w(t_{0}))=\sqrt{V(t)}-\frac{\sigma}%
{2}(w(t)-w(t_{0})),\ t_{0}\leq t<t_{0}+\tau,
\]
the equation
\begin{align}
\frac{dY}{dt}  &  =\frac{\alpha}{Y+\frac{\sigma}{2}(w(t)-w(t_{0}))}-\frac
{k}{2}(Y+\frac{\sigma}{2}(w(t)-w(t_{0}))),\ \label{DS15}\\
Y(t_{0};t_{0})  &  =\sqrt{V(t_{0})},\ t_{0}\leq t<t_{0}+\tau,\nonumber
\end{align}
with $\alpha$ given by (\ref{DS8a}). Clearly,%
\begin{equation}
V(t)=(Y(t;t_{0})+\frac{\sigma}{2}(w(t)-w(t_{0})))^{2},\ t_{0}\leq t<t_{0}%
+\tau. \label{DS17}%
\end{equation}

\medskip

\section{Auxiliary propositions}

\textbf{3.1} Let us consider in view of (\ref{DS15}) solutions of the ordinary
differential equations%
\begin{equation}
\frac{dy^{0}}{dt}=\frac{\alpha}{y^{0}}-\frac{k}{2}y^{0},\ y^{0}(t_{0}%
)=y_{0}>0,\ t\geq t_{0}\geq0, \label{OS1}%
\end{equation}
which are given by%
\begin{equation}
y^{0}(t)=y_{t_{0},y_{0}}^{0}(t)=[y_{0}^{2}e^{-k(t-t_{0})}+\frac{2\alpha}%
{k}(1-e^{-k(t-t_{0})})]^{1/2},\ t\geq t_{0}. \label{OS2}%
\end{equation}

In the case $\alpha>0,$ i.e., $4k\lambda>\sigma^{2},$ we have: if $y_{0}%
>\sqrt{2\alpha/k}$ then $y_{t_{0},y_{0}}^{0}(t)\downarrow\sqrt{2\alpha/k}$ as
$t\rightarrow\infty$ and if $0<y_{0}<\sqrt{2\alpha/k}$ then $y_{t_{0},y_{0}%
}^{0}(t)\uparrow\sqrt{2\alpha/k}$ as $t\rightarrow\infty.$ Further
$y^{0}(t)=\sqrt{2\alpha/k}$ is a solution of (\ref{OS1}).

In the case $\alpha=0,$ the solution $y_{t_{0},y_{0}}^{0}(t)\downarrow0$ under
$t\rightarrow\infty$ for any $y_{0}>0.$ We note that the case $\alpha\geq0$ is
more general than the case $2k\lambda\geq\sigma^{2}$ (we recall that in the
latter case $V(t)>0$,$\ t\geq t_{0}$).

In the case $\alpha<0,$ the solution $y_{t_{0},y_{0}}^{0}(t)$ is convexly
decreasing under not too large $y_{0}$. It attains zero at the moment $\bar
{t}$ given by%
\begin{equation}
\bar{t}=t_{0}+\frac{1}{k}\ln\frac{y_{0}^{2}-2\alpha/k}{-2\alpha/k} \label{OS3}%
\end{equation}
and $y_{t_{0},y_{0}}^{0\prime}(\bar{t})=-\infty.$

\medskip In what follows we deal with the case
\begin{equation}
\alpha=\dfrac{4k\lambda-\sigma^{2}}{8}\geq0. \label{OS05}%
\end{equation}
\textbf{3.2}. Our next goal is to obtain estimates for solutions of the
equation
\begin{equation}
\frac{dy}{dt}=\frac{\alpha}{y+\frac{\sigma}{2}\varphi(t)}-\frac{k}{2}%
(y+\frac{\sigma}{2}\varphi(t)),\ y(t_{0})=y_{0},\ t_{0}\leq t\leq t_{0}%
+\theta, \label{OS4}%
\end{equation}
(cf. (\ref{DS15}) ) for a given continuous function $\varphi(t).$\medskip

\begin{lemma}
\label{Lemma 2} {Let} $\alpha\geq0.${ Let }$y^{i}(t),\ i=1,2,${ be two
solutions of (\ref{OS4}) such that }$y^{i}(t)+\frac{\sigma}{2}\varphi(t)>0$
{on} $[t_{0},t_{0}+\theta],$ for some $\theta$ with $0\leq\theta\leq T.$
\textit{Then}%
\begin{equation}
\left\vert y^{2}(t)-y^{1}(t)\right\vert \leq\left\vert y^{2}(t_{0}%
)-y^{1}(t_{0})\right\vert ,\ t_{0}\leq t\leq t_{0}+\theta. \label{OS06}%
\end{equation}

\end{lemma}

\begin{proof}
We have%
\begin{gather}
d(y^{2}(t)-y^{1}(t))^{2}=2(y^{2}(t)-y^{1}(t))\label{val}\\
\times\left(  \frac{\alpha}{y^{2}(t)+\frac{\sigma}{2}\varphi(t)}-\frac{k}%
{2}(y^{2}(t)+\frac{\sigma}{2}\varphi(t))-\frac{\alpha}{y^{1}(t)+\frac{\sigma
}{2}\varphi(t)}+\frac{k}{2}(y^{1}(t)+\frac{\sigma}{2}\varphi(t))\right)
dt.\nonumber
\end{gather}
From here%
\begin{gather*}
(y^{2}(t)-y^{1}(t))^{2}=(y^{2}(t_{0})-y^{1}(t_{0}))^{2}\\
+2\int_{t_{0}}^{t}[-\alpha\frac{(y^{2}(s)-y^{1}(s))^{2}}{(y^{1}(s)+\frac
{\sigma}{2}\varphi(s))(y^{2}(s)+\frac{\sigma}{2}\varphi(s))}-\frac{k}{2}%
(y^{2}(s)-y^{1}(s))^{2}]ds\\
\leq(y^{2}(t_{0})-y^{1}(t_{0}))^{2},
\end{gather*}
whence (\ref{OS06}) follows.
\end{proof}

\begin{remark}
\label{L2*} \ It is known that for $\delta>1$ the Bessel process BES$^{\delta
}$ has the representation%
\[
Z(t)=Z(0)+\frac{\delta-1}{2}\int_{0}^{t}\frac{1}{Z(s)}ds\,+W(t),\text{
\ \ }0\leq t<\infty,
\]
where $W$ is standard Brownian Motion, $Z(t)\geq0$ a.s., and that in
particular $E\int_{0}^{t}\frac{1}{Z(s)}ds<\infty.$ (See \cite{RY}; for
$\delta\leq1$ the representation of BES$^{\delta}$ is less simple and involves
the concept of local time.) From this fact it is not difficult to show that
for $\alpha>0$ the solution of (\ref{DS8}) may be represented as
\[
U(t)=U(t_{0})+\int_{t_{0}}^{t}(\frac{\alpha}{U(s)}-\frac{k}{2}U(s))ds+\frac
{\sigma}{2}\left(  w(t)-w(t_{0})\right)  ,\ U(0)>0,\text{ \ \ }t_{0}\leq
t<\infty.
\]
Thus, with $Y(t)=U(t)-\frac{\sigma}{2}\left(  w(t)-w(t_{0})\right)  ,$ it
holds that%
\[
Y(t)=Y(t_{0})+\int_{t_{0}}^{t}\left(  \frac{\alpha}{Y(s)+\frac{\sigma}%
{2}\left(  w(s)-w(t_{0})\right)  }-\frac{k}{2}\left(  Y(s)+\frac{\sigma}%
{2}\left(  w(s)-w(t_{0})\right)  \right)  \right)  ds,
\]
for $Y(0)=U(0)>0,$ $0\leq t<\infty,$ and that in particular $Y$ is continuous
and of bounded variation. From this it follows that (\ref{val}) holds for $t_0\le t\le t_0+T$ when
$\alpha>0$ and $\varphi(t)=w(t)-w(t_{0})$ is an arbitrary Brownian trajectory, 
and then inequality (\ref{OS06}) in Lemma \ref{Lemma 2} goes through for
$\theta=T.$
\end{remark}
\textbf{3.3} Now consider (\ref{OS4}) for a continuous function $\varphi$
satisfying%
\begin{equation}
\left\vert \varphi(t)\right\vert \leq r,\ t_{0}\leq t\leq t_{0}+\theta\leq
t_{0}+T, \label{OS5}%
\end{equation}
for some $r>0$ and $0\leq\theta\leq T.$ Along with (\ref{OS1}), (\ref{OS4})
with (\ref{OS5}), we further consider the equations%
\begin{align}
\frac{dy}{dt}  &  =\frac{\alpha}{y+\frac{\sigma}{2}r}-\frac{k}{2}%
(y+\frac{\sigma}{2}r),\ y(t_{0})=y_{0},\label{OS6}\\
\frac{dy}{dt}  &  =\frac{\alpha}{y-\frac{\sigma}{2}r}-\frac{k}{2}%
(y-\frac{\sigma}{2}r),\ y(t_{0})=y_{0}. \label{OS7}%
\end{align}
Let us assume that $y_{0}\geq\sigma r>0,$ and consider an $\eta>0,$ to be
specified below, that satisfies
\begin{equation}
y_{0}\geq\eta\geq\sigma r>0. \label{OS8}%
\end{equation}
The solutions of (\ref{OS1}), (\ref{OS4}) with (\ref{OS5}), (\ref{OS6}), and
(\ref{OS7}) are denoted by $y^{0}(t),\ y(t),\ y^{-}(t)$, and $y^{+}(t),$
respectively, where $y^{0}(t)$ is given by (\ref{OS2}). By using (\ref{OS2})
we derive straightforwardly that%
\begin{align}
y^{-}(t)  &  =[(y_{0}+\frac{\sigma}{2}r)^{2}e^{-k(t-t_{0})}+\frac{2\alpha}%
{k}(1-e^{-k(t-t_{0})})]^{1/2}-\frac{\sigma}{2}r,\ t_{0}\leq t\leq t_{0}%
+\theta,\label{OS9}\\
y^{+}(t)  &  =[(y_{0}-\frac{\sigma}{2}r)^{2}e^{-k(t-t_{0})}+\frac{2\alpha}%
{k}(1-e^{-k(t-t_{0})})]^{1/2}+\frac{\sigma}{2}r,\ t_{0}\leq t\leq t_{0}%
+\theta. \label{OS10}%
\end{align}
Note that $y^{-}(t)+\sigma r/2>0$ and $y^{+}(t)>\sigma r/2,$ $t_{0}\leq t\leq
t_{0}+\theta.$ Due to the comparison theorem for ODEs (see, e.g., \cite{H},
Ch. 3), the inequality%
\[
\frac{\alpha}{y+\frac{\sigma}{2}r}-\frac{k}{2}(y+\frac{\sigma}{2}r)\leq
\frac{\alpha}{y+\frac{\sigma}{2}\varphi(t)}-\frac{k}{2}(y+\frac{\sigma}%
{2}\varphi(t))\leq\frac{\alpha}{y-\frac{\sigma}{2}r}-\frac{k}{2}%
(y-\frac{\sigma}{2}r),
\]
which is fulfilled in view of (\ref{OS5}) for $y>\sigma r/2,$ then implies
that%
\begin{equation}
y^{-}(t)\leq y(t)\leq y^{+}(t),\ t_{0}\leq t\leq t_{0}+\theta. \label{OS11}%
\end{equation}
The same inequality holds for $y(t)$ replaced by $y^{0}(t).$ We thus get%
\begin{equation}
\left\vert y(t)-y^{0}(t)\right\vert \leq y^{+}(t)-y^{-}(t),\ t_{0}\leq t\leq
t_{0}+\theta. \label{OS12}%
\end{equation}

\begin{proposition}
\label{Proposition 3} {Let }$\alpha=\dfrac{4k\lambda-\sigma^{2}}{8}\geq0,${
the inequalities (\ref{OS5}) and (\ref{OS8}) be fulfilled for a fixed }%
$\eta>0,$ and {let }$\theta\leq T.$ We then have
\begin{align}
\left\vert y(t)-y^{0}(t)\right\vert  &  \leq Cr(t-t_{0})\leq Cr\theta
,\ t_{0}\leq t\leq t_{0}+\theta,\text{ \ \ with}\label{OS13}\\
C  &  =\dfrac{\sigma k}{2}+\dfrac{4\alpha\sigma}{3\eta^{2}}e^{\frac{k}{2}%
T}.\nonumber
\end{align}
{In particular, }$C$ \textit{is independent of}$\ t_{0},\ y_{0},$ and $r$
(provided {(\ref{OS8}) holds}).
\end{proposition}

\begin{proof}
We estimate the difference $y^{+}(t)-y^{-}(t).$ It holds%
\begin{gather}
y^{+}(t)=z^{-}(t)+\frac{\sigma}{2}r,\ y^{-}(t)=z^{+}(t)-\frac{\sigma}%
{2}r,\ \nonumber\\
y^{+}(t)-y^{-}(t)=\sigma r-(z^{+}(t)-z^{-}(t)), \label{OS14}%
\end{gather}
where%
\[
z^{\pm}(t)=[(y_{0}\pm\frac{\sigma}{2}r)^{2}e^{-k(t-t_{0})}+\frac{2\alpha}%
{k}(1-e^{-k(t-t_{0})})]^{1/2}.
\]
Further,%
\begin{equation}
z^{+}(t)-z^{-}(t)=\frac{(z^{+}(t))^{2}-(z^{-}(t))^{2}}{z^{+}(t)+z^{-}%
(t)}=\frac{2y_{0}\sigma re^{-k(t-t_{0})}}{z^{+}(t)+z^{-}(t)}. \label{OS15}%
\end{equation}
Using the inequality $(a^{2}+b)^{1/2}\leq a+b/2a$ for any $a>0$ and $b\geq0,$
we get
\begin{align*}
z^{+}(t)  &  \leq(y_{0}+\frac{\sigma}{2}r)e^{-\frac{k}{2}(t-t_{0})}%
+\frac{\alpha}{k}\frac{(1-e^{-k(t-t_{0})})}{(y_{0}+\frac{\sigma}{2}%
r)e^{-\frac{k}{2}(t-t_{0})}},\\
z^{-}(t)  &  \leq(y_{0}-\frac{\sigma}{2}r)e^{-\frac{k}{2}(t-t_{0})}%
+\frac{\alpha}{k}\frac{(1-e^{-k(t-t_{0})})}{(y_{0}-\frac{\sigma}{2}%
r)e^{-\frac{k}{2}(t-t_{0})}},
\end{align*}
whence%
\[
z^{+}(t)+z^{-}(t)\leq2y_{0}e^{-\frac{k}{2}(t-t_{0})}+\frac{\alpha}{k}%
\frac{(1-e^{-k(t-t_{0})})}{e^{-\frac{k}{2}(t-t_{0})}}\frac{2y_{0}}{(y_{0}%
^{2}-\dfrac{\sigma^{2}}{4}r^{2})}.
\]
Therefore%
\[
\frac{1}{z^{+}(t)+z^{-}(t)}\geq\frac{1}{2y_{0}e^{-\frac{k}{2}(t-t_{0})}%
}\left(  1-\frac{\alpha}{k(y_{0}^{2}-\dfrac{\sigma^{2}}{4}r^{2})}%
(e^{k(t-t_{0})}-1)\right)  .
\]
From (\ref{OS15}) we have that%
\[
z^{+}(t)-z^{-}(t)\geq\sigma re^{-\frac{k}{2}(t-t_{0})}\left(  1-\frac{\alpha
}{k(y_{0}^{2}-\dfrac{\sigma^{2}}{4}r^{2})}(e^{k(t-t_{0})}-1)\right)
\]
and so due to (\ref{OS14}) we get%
\[
0\leq y^{+}(t)-y^{-}(t)\leq\sigma r(1-e^{-\frac{k}{2}(t-t_{0})})+\frac
{\alpha\sigma r}{k(y_{0}^{2}-\dfrac{\sigma^{2}}{4}r^{2})}(e^{\frac{k}%
{2}(t-t_{0})}-e^{-\frac{k}{2}(t-t_{0})}).
\]
Since $1-e^{-q\vartheta}\leq q\vartheta$ for any $q\geq0,\ \vartheta\geq0,$
and $y_{0}^{2}-\dfrac{\sigma^{2}}{4}r^{2}\geq\dfrac{3}{4}\eta^{2}$ under
(\ref{OS8}), we obtain
\[
0\leq y^{+}(t)-y^{-}(t)\leq\frac{\sigma rk}{2}(t-t_{0})+\frac{4\alpha\sigma
r}{3k\eta^{2}}e^{\frac{k}{2}(t-t_{0})}k(t-t_{0}).
\]
From this and (\ref{OS12}), (\ref{OS13}) follows with $C=\dfrac{\sigma k}%
{2}+\dfrac{4\alpha\sigma}{3\eta^{2}}e^{\frac{k}{2}T}.$
\end{proof}

\begin{corollary}
\label{corr} Under the assumptions of Proposition~\ref{Proposition 3}, we get
by taking $\eta=y_{0},$%
\begin{align*}
\left\vert y(t)-y^{0}(t)\right\vert  &  \leq\left(  \dfrac{\sigma k}{2}%
+\dfrac{4\alpha\sigma}{3y_{0}^{2}}e^{\frac{k}{2}T}\right)  r\theta,\\
&  :=\left(  D_{1}+\frac{D_{2}}{y_{0}^{2}}\right)  r\theta,\text{ \ \ }%
\ t_{0}\leq t\leq t_{0}+\theta,
\end{align*}
where $D_{1}$ and $D_{2}$ only depend on the parameters of the CIR process
under consideration and the time horizon~$T.$
\end{corollary}

\section{One-step approximation}

Let us suppose that for $t_{m},\ t_{0}\leq t_{m}<t_{0}+T,$ $V(t_{m})$ is known
exactly. In fact, $t_{m}$ may be considered as a realization of a certain
stopping time. Consider $Y=Y(t;t_{m})$ on some interval $[t_{m},t_{m}%
+\theta_{m}]$ with $y_{m}:=Y(t_{m};t_{m})=\sqrt{V(t_{m})},$ given by the ODE
(cf. (\ref{DS15})),%
\begin{align}
\frac{dY}{dt}  &  =\frac{\alpha}{Y+\frac{\sigma}{2}(w(t)-w(t_{m}))}-\frac
{k}{2}(Y+\frac{\sigma}{2}(w(t)-w(t_{m}))),\ \label{OS19}\\
Y(t_{m};t_{m})  &  =\sqrt{V(t_{m})},\ t_{m}\leq t\leq t_{m}+\theta
_{m}.\nonumber
\end{align}
Assume that%
\begin{equation}
y_{m}=\sqrt{V(t_{m})}\geq\sigma r. \label{OS190}%
\end{equation}
Due to (\ref{DS17}), the solution $V(t)$ of (\ref{In1}) on $[t_{m}%
,t_{m}+\theta_{m}]$ is obtained via%
\begin{equation}
\sqrt{V(t)}=Y(t;t_{m})+\frac{\sigma}{2}(w(t)-w(t_{m})),\ t_{m}\leq t\leq
t_{m}+\theta_{m}. \label{OS019}%
\end{equation}
Though equation (\ref{OS19}) is (just) an ODE, it is not easy to solve it
numerically in a straightforward way because of the non-smoothness of $w(t).$
We are here going to construct an approximation $y^{m}(t)$ of $Y(t;t_{m})$ via
Proposition~\ref{Proposition 3}. To this end we simulate the point
$(t_{m}+\theta_{m},w(t_{m}+\theta_{m})-w(t_{m}))$ by simulating $\theta_{m}$
as being the first-passage (stopping) time of the Wiener process
$w(t)-w(t_{m}),$ $t\geq t_{m},$ to the boundary of the interval $[-r,r].$ So,
$\left\vert w(t)-w(t_{m})\right\vert \leq r$ for $t_{m}\leq t\leq t_{m}%
+\theta_{m}$ and, moreover, the random variable $w(t_{m}+\theta_{m}%
)-w(t_{m}),$ which equals either $-r$ or $+r$ with probability $1/2,$ is
independent of the stopping time $\theta_{m}.$ A method for simulating the
stopping time $\theta_{m}$ is given in Subsection~\ref{sim} below. Proposition
\ref{Proposition 3} and Corollary \ref{corr} then yield,%
\begin{align}
\left\vert Y(t;t_{m})-y^{m}(t)\right\vert  &  \leq\left(  D_{1}+\frac{D_{2}%
}{y_{m}^{2}}\right)  r\left(  t_{m+1}-t_{m}\right)  ,\text{ \ \ }t_{m}\leq
t\leq t_{m+1}\text{ \ \ with}\label{OS16}\\
t_{m+1}  &  :=\min(t_{m}+\theta_{m},t_{0}+T),\nonumber
\end{align}
where $y^{m}(t)$ is the solution of the problem
\[
\frac{dy^{m}}{dt}=\frac{\alpha}{y^{m}}-\frac{k}{2}y^{m},\ y^{m}(t_{m}%
)=Y(t_{m};t_{m})=\sqrt{V(t_{m})}%
\]
that is given by (\ref{OS2}) with $(t_{m},y_{m})=(t_{m},\sqrt{V(t_{m})}).$ We
so have,
\[
\sqrt{V(t)}=Y(t;t_{m})+\frac{\sigma}{2}(w(t)-w(t_{m}))=y^{m}(t)+\frac{\sigma
}{2}(w(t)-w(t_{m}))+\rho^{m}(t),
\]
where due to (\ref{OS16}),%
\begin{equation}
\left\vert \rho^{m}(t)\right\vert \leq\left(  D_{1}+\frac{D_{2}}{y_{m}^{2}%
}\right)  r\left(  t_{m+1}-t_{m}\right)  ,\text{ \ \ }t_{m}\leq t\leq t_{m+1}.
\label{OS017}%
\end{equation}
We next introduce the one-step approximation $\sqrt{\overline{V}(t)}$ of
$\sqrt{V(t)}$ on $[t_{m},t_{m+1}]$ by%
\begin{equation}
\sqrt{\overline{V}(t)}:=y^{m}(t)+\frac{\sigma}{2}(w(t)-w(t_{m})),\text{
\ \ }t_{m}\leq t\leq t_{m+1}. \label{OS17}%
\end{equation}
Since $\left\vert w(t_{m+1})-w(t_{m})\right\vert =r$ if $t_{m+1}=t_{m}%
+\theta_{m}<t_{0}+T,$ and $\left\vert w(t_{m+1})-w(t_{m})\right\vert \leq r$
if $t_{m+1}=t_{0}+T,$ the one-step approximation (\ref{OS17}) for $t=t_{m+1}$
is given by%
\begin{gather}
\sqrt{\overline{V}(t_{m+1})}:=y^{m}(t_{m+1})+\frac{\sigma}{2}(w(t_{m+1}%
)-w(t_{m}))=\label{OS180}\\
y^{m}(t_{m+1})+\frac{\sigma}{2}\cdot\left\{
\begin{tabular}
[c]{l}%
$r\xi_{m}$ \ \ with $P(\xi_{m}=\pm1)=1/2,$ if $t_{m+1}=t_{m}+\theta_{m}%
<t_{0}+T,$\\
$\zeta_{m}$ \ \ \ if $t_{m+1}=t_{0}+T,$%
\end{tabular}
\ \ \ \ \right. \nonumber
\end{gather}
with $\zeta_{m}=w(t_{0}+T)-w(t_{m})$ being drawn from the distribution of%
\begin{equation}
W_{t_{0}+T-t_{m}}\text{ \ \ conditional on \ \ }\max_{0\leq s\leq
t_{0}+T-t_{m}}\,\left\vert W_{s}\right\vert \leq r, \label{con}%
\end{equation}
where $W$ is an independent standard Brownian motion. For details see
Subsection~\ref{sim} below. We so have the following theorem.

\begin{theorem}
\label{Theorem 4} For the one-step approximation $\overline{V}(t_{m+1})$ due
to the exact starting value $\overline{V}(t_{m})$ $=$ $V(t_{m})$ $=$
$y_{m}^{2},$ we have the one step error%
\begin{equation}
\left\vert \sqrt{V(t_{m+1})}-\sqrt{\overline{V}(t_{m+1})}\right\vert
\leq\left(  D_{1}+\frac{D_{2}}{V(t_{m})}\right)  r\left(  t_{m+1}%
-t_{m}\right)  . \label{OS018}%
\end{equation}

\end{theorem}

\subsection{Simulation of $\theta_{m}$ and $\zeta_{m}$}

\label{sim}

For simulating $\theta_{m}$ we utilize the distribution function%
\[
\mathcal{P}(t):=P(\tau<t),
\]
where $\tau$ is the first-passage time of the Wiener process $W(t)$ to the
boundary of the interval $[-1,1].$ A very accurate approximation
$\mathcal{\tilde{P}}(t)$ of $\mathcal{P}(t)$ is the following one:%
\[
\mathcal{P}(t)\simeq\mathcal{\tilde{P}}(t)=\int_{0}^{t}\mathcal{\tilde{P}%
}^{\prime}(s)ds\text{ \ \ with}%
\]%
\[
\mathcal{\tilde{P}}^{\prime}(t)=\left\{
\begin{array}
[c]{c}%
\dfrac{2}{\sqrt{2\pi t^{3}}}(e^{-\dfrac{1}{2t}}-3e^{-\dfrac{9}{2t}%
}+5e^{-\dfrac{25}{2t}}),\ 0<t\leq\dfrac{2}{\pi},\\
\dfrac{\pi}{2}(e^{-\dfrac{\pi^{2}t}{8}}-3e^{-\dfrac{9\pi^{2}t}{8}}%
+5e^{-\dfrac{25\pi^{2}t}{8}}),\ t>\dfrac{2}{\pi},
\end{array}
\right.
\]
and it holds%
\[
\sup_{t\geq0}\left\vert \mathcal{\tilde{P}}^{\prime}(t)-\mathcal{P}^{\prime
}(t)\right\vert \leq2.13\times10^{-16},\text{ \ and \ }\sup_{t\geq0}\left\vert
\mathcal{\tilde{P}}(t)-\mathcal{P}(t)\right\vert \leq7.04\times10^{-18},
\]
(see for details \cite{MT1}, Ch. 5, Sect. 3 and Appendix A3 ). Now simulate a
random variable $U$ uniformly distributed on $[0,1],$ Then compute
$\tau=\mathcal{P}^{-1}(U)$ which is distributed according to $\mathcal{P}.$
That is, we have to solve the equation $\mathcal{\tilde{P}}(\tau)=U,$ for
instance by Newton's method or any other efficient solving routine. Next set
$\theta_{m}=r^{2}\tau_{m}.$

For simulating $\zeta_{m}$ in (\ref{OS180}) we observe that (\ref{con}) is
equivalent with%
\[
rW_{r^{-2}\left(  t_{0}+T-t_{m}\right)  }\text{ \ \ conditional on \ \ }%
\max_{0\leq u\leq r^{-2}\left(  t_{0}+T-t_{m}\right)  }\,\left\vert
W_{u}\right\vert \leq1.
\]
We next sample $\vartheta$ from the distribution function $\mathcal{Q}%
(x;r^{-2}\left(  t_{0}+T-t_{m}\right)  ),$ where $\mathcal{Q}(x;t)$ is the
known conditional distribution function (see \cite{MT1}, Ch. 5, Sect. 3)%
\begin{equation}
\mathcal{Q}(x;t):=P(\left.  W(t)<x\text{ }\right\vert \text{ max}_{0\leq s\leq
t}\left\vert W(s)\right\vert <1),\text{ \ \ }-1\leq x\leq1, \label{frakQ}%
\end{equation}
and set $\zeta_{m}=r\vartheta.$ The simulation of the last step looks rather
complicated and may be computationally expensive. However it is possible to
take for $w(t_{0}+T)-w(t_{\nu})$ simply any value between $-r$ and $r,$ e.g.
zero. This may enlarge the one-step error on the last step but does not
influence the convergence order of the elaborated method. Indeed, if we set
$w(t_{0}+T)-w(t_{\nu})$ to be zero, for instance, on the last step, we get
$\sqrt{\overline{V}(t_{0}+T)}=y^{\nu}(t_{0}+T)$ instead of (\ref{OS180}), and
\begin{equation}
\left\vert \sqrt{V(t_{0}+T)}-\sqrt{\overline{V}(t_{0}+T)}\right\vert \leq
r\sum_{m=0}^{\nu}\left(  D_{1}+\frac{D_{2}}{\overline{V}(t_{m})}\right)
\left(  t_{m+1}-t_{m}\right)  +\sigma r, \label{mod}%
\end{equation}

\begin{remark}
We have in any step $E\theta_{n}=r^{2},$ the random number of steps before
reaching $t_{0}+T,$ say $\nu+1,$ is finite with probability one, and
$E\nu=O(1/r^{2}).$ For details see \cite{MT1}, Ch. 5, Lemma 1.5. In a
heuristic sense this means that, if we have convergence of order $O(r),$ we
obtain accuracy $O(\sqrt{h}),$ for an (expected) number of steps $O(1/h)$
similar to the standard Euler scheme.
\end{remark}

\section{Convergence theorem}

In this section we develop a scheme that generates approximations
$\sqrt{\overline{V}(t_{0})}=\sqrt{V(t_{0})},$ $\sqrt{\overline{V}(t_{1})},$
$...$ $,\sqrt{\overline{V}(t_{n+1})},$ where $n=0,1,2,...,$ and $t_{1}%
,...,t_{n+1}$ are realizations of a sequence of stopping times, and show that
the global error in approximation $\sqrt{\overline{V}(t_{n+1})}$ is in fact an
aggregated sum of local errors, i.e.,
\[
r\sum_{m=0}^{n}\left(  D_{1}+\frac{D_{2}}{\overline{V}(t_{m})}\right)  \left(
t_{m+1}-t_{m}\right)  \leq rT\left(  D_{1}+\frac{D_{2}}{\eta_{n}^{2}}\right)
,
\]
with $y_{m}=$\ $\sqrt{\overline{V}(t_{m})},$ provided that $y_{m}\geq\sigma r$
for $m=0,...,n,$ and so $\eta_{n}:=\min_{0\leq m\leq n}y_{m}\geq\sigma r.$

Let us now describe an algorithm for the solution of (\ref{In1}) on the
interval $[t_{0},t_{0}+T]$ in the case $\alpha\geq0.$ Suppose we are given
$V(t_{0})$ and $r$ such that%
\[
\sqrt{V(t_{0})}\geq\sigma r.
\]
For the initial step we use the one-step approximation according to the
previous section and thus obtain (see (\ref{OS180}) and (\ref{OS018}))%
\begin{align*}
\sqrt{\overline{V}(t_{1})}  &  =y^{0}(t_{1})+\frac{\sigma}{2}(w(t_{1}%
)-w(t_{0})),\\
\sqrt{V(t_{1})}  &  =\sqrt{\overline{V}(t_{1})}+\rho^{0}(t_{1}),
\end{align*}
where
\begin{equation}
\left\vert \rho^{0}(t_{1})\right\vert \leq\left(  D_{1}+\frac{D_{2}}{V(t_{0}%
)}\right)  r\left(  t_{1}-t_{0}\right)  =:C_{0}r(t_{1}-t_{0}). \label{ins}%
\end{equation}
Suppose that
\[
\sqrt{\overline{V}(t_{1})}\geq\sigma r.
\]
We then go to the next step and consider the expression
\begin{equation}
\sqrt{V(t)}=Y(t;t_{1})+\frac{\sigma}{2}(w(t)-w(t_{1})), \label{CT2}%
\end{equation}
where $Y(t;t_{1})$ is the solution of the problem (see (\ref{OS19}))%
\begin{align}
\frac{dY}{dt}  &  =\frac{\alpha}{Y+\frac{\sigma}{2}(w(t)-w(t_{1}))}-\frac
{k}{2}(Y+\frac{\sigma}{2}(w(t)-w(t_{1}))),\ \label{CT3}\\
Y(t_{1};t_{1})  &  =\sqrt{V(t_{1})},\ t_{1}\leq t\leq t_{1}+\theta
_{1}.\nonumber
\end{align}
Now, in contrast to the initial step, the value $\sqrt{V(t_{1})}$ is unknown
and we are forced to use $\sqrt{\overline{V}(t_{1})}$ instead. Therefore we
introduce $\overline{Y}(t;t_{1})$ as the solution of the equation (\ref{CT3})
with initial value $\overline{Y}(t_{1};t_{1})=\sqrt{\overline{V}(t_{1})}.$
From the previous step we have that $\left\vert Y(t_{1};t_{1})-\overline
{Y}(t_{1};t_{1})\right\vert =\left\vert \sqrt{V(t_{1})}-\sqrt{\overline
{V}(t_{1})}\right\vert =$ $\left\vert \rho^{0}(t_{1})\right\vert \leq
C_{0}r(t_{1}-t_{0}).$ Hence, due to Lemma~\ref{Lemma 2},
\begin{equation}
\left\vert Y(t;t_{1})-\overline{Y}(t;t_{1})\right\vert \leq\rho^{0}(t_{1})\leq
C_{0}r(t_{1}-t_{0}),\ t_{1}\leq t\leq t_{1}+\theta_{1}. \label{OS18}%
\end{equation}
Let $\theta_{1}$ be the first-passage time of the Wiener process
$w(t_{1}+\cdot)-w(t_{1})$ to the boundary of the interval $[-r,r].$ If
$t_{1}+\theta_{1}<t_{0}+T$ then set $t_{2}:=t_{1}+\theta_{1},$ else set
$t_{2}:=t_{0}+T.$ In order to approximate $\overline{Y}(t;t_{1})$ for
$t_{1}\leq t\leq t_{2}$ let us consider along with equation (\ref{CT3}) the
equation%
\[
\frac{dy^{1}}{dt}=\frac{\alpha}{y^{1}}-\frac{k}{2}y^{1},\ y^{1}(t_{1}%
)=\overline{Y}(t_{1};t_{1})=\sqrt{\overline{V}(t_{1})}.
\]
Due to Proposition 3 and Corollary \ref{corr} it holds that%
\begin{equation}
\left\vert \overline{Y}(t;t_{1})-y^{1}(t)\right\vert \leq\left(  D_{1}%
+\frac{D_{2}}{\overline{V}(t_{1})}\right)  r\left(  t_{2}-t_{1}\right)
=:C_{1}r\left(  t_{2}-t_{1}\right)  ,\ \ t_{1}\leq t\leq t_{2}. \label{R1}%
\end{equation}
and so by (\ref{OS18}) we have%
\begin{equation}
\left\vert Y(t;t_{1})-y^{1}(t)\right\vert \leq r(C_{0}(t_{1}-t_{0}%
)+C_{1}\left(  t_{2}-t_{1}\right)  ),\ \ \ t_{1}\leq t\leq t_{2}. \label{R2}%
\end{equation}
We also have (see (\ref{CT2}))%
\begin{equation}
\sqrt{V(t)}=Y(t;t_{1})+\frac{\sigma}{2}(w(t)-w(t_{1}))=y^{1}(t)+\frac{\sigma
}{2}(w(t)-w(t_{1}))+R^{1}(t), \label{R3}%
\end{equation}
where%
\begin{equation}
\left\vert R^{1}(t)\right\vert \leq r(C_{0}(t_{1}-t_{0})+C_{1}\left(
t_{2}-t_{1}\right)  ),\ t_{1}\leq t\leq t_{2}. \label{R4}%
\end{equation}
We so define the approximation%
\begin{align}
\sqrt{\overline{V}(t)}  &  :=y^{1}(t)+\frac{\sigma}{2}(w(t)-w(t_{1})),\text{
\ \ that satisfies}\label{R5}\\
\sqrt{V(t)}  &  =\sqrt{\overline{V}(t)}+R^{1}(t)\ ,\text{ \ \ }t_{1}\leq t\leq
t_{2}.\text{ \ } \label{R6}%
\end{align}
and then set%
\begin{gather}
\sqrt{\overline{V}(t_{2})}=y^{1}(t_{2})+\frac{\sigma}{2}(w(t_{2}%
)-w(t_{1}))=\label{R7}\\
y^{1}(t_{2})+\frac{\sigma}{2}\cdot\left\{
\begin{tabular}
[c]{l}%
\ $r\xi_{1}$ \ \ with $P(\xi_{1}=\pm1)=1/2,$ if $t_{2}=t_{1}+\theta_{1}%
<t_{0}+T,$\\
$\zeta_{1}$ \ \ \ if $t_{2}=t_{0}+T,$%
\end{tabular}
\ \ \ \ \right.  ,\nonumber
\end{gather}
cf. (\ref{OS180}) and (\ref{con}). We thus end up with a next approximation
$\sqrt{\overline{V}(t_{2})}$ such that%
\begin{equation}
\left\vert \sqrt{V(t_{2})}-\sqrt{\overline{V}(t_{2})}\right\vert =\left\vert
R^{1}(t_{2})\right\vert \leq r(C_{0}(t_{1}-t_{0})+C_{1}\left(  t_{2}%
-t_{1}\right)  ). \label{R8}%
\end{equation}

\bigskip From the above description it is obvious how to proceed analogously
given a generic approximation sequence of approximations $\sqrt{\overline
{V}(t_{m})},$ $m=0,1,2,...,n,$ with $\overline{V}(t_{0})=V(t_{0}),$ that
satisfies by assumption%
\begin{align}
\sqrt{\overline{V}(t_{m})}  &  \geq\sigma r,\text{ \ \ for }m=0,...,n,\text{
\ \ and}\label{CT1}\\
\text{ \ \ }\left\vert \sqrt{V(t_{n})}-\sqrt{\overline{V}(t_{n})}\right\vert
&  \leq r\sum_{m=0}^{n-1}\left(  D_{1}+\frac{D_{2}}{\overline{V}(t_{m}%
)}\right)  \left(  t_{m+1}-t_{m}\right) \label{CT02}\\
&  =:r\sum_{m=0}^{n-1}C_{m}\left(  t_{m+1}-t_{m}\right)  .\nonumber
\end{align}
Indeed, consider the expression
\[
\sqrt{V(t)}=Y(t;t_{n})+\frac{\sigma}{2}(w(t)-w(t_{n})),
\]
where $Y(t;t_{n})$ is the solution of the problem
\begin{align}
\frac{dY}{dt}  &  =\frac{\alpha}{Y+\frac{\sigma}{2}(w(t)-w(t_{n}))}-\frac
{k}{2}(Y+\frac{\sigma}{2}(w(t)-w(t_{n}))),\ \label{CT03}\\
Y(t_{n};t_{n})  &  =\sqrt{V(t_{n})},\ t_{n}\leq t\leq t_{n}+\theta
_{n},\nonumber
\end{align}
for a $\theta_{n}>0$ to be determined$.$ Since $\sqrt{V(t_{n})}$ is unknown we
consider $\overline{Y}(t;t_{n})$ as the solution of the equation (\ref{CT03})
with initial value $\overline{Y}(t_{n};t_{n})=\sqrt{\overline{V}(t_{n})}.$ Due
to (\ref{CT02}) and Lemma 2 again, we have%
\[
\left\vert Y(t;t_{n})-\overline{Y}(t;t_{n})\right\vert \leq r\sum_{m=0}%
^{n-1}C_{m}\left(  t_{m+1}-t_{m}\right)  ,\ t_{n}\leq t\leq t_{n}+\theta_{n}.
\]
In order to approximate $\overline{Y}(t;t_{n})$ for $t_{n}\leq t\leq
t_{n}+\theta_{n},$ we consider the equation%
\begin{equation}
\frac{dy^{n}}{dt}=\frac{\alpha}{y^{n}}-\frac{k}{2}y^{n},\ y^{n}(t_{n}%
)=\overline{Y}(t_{n};t_{n})=\sqrt{\overline{V}(t_{n})}. \label{R09}%
\end{equation}
By repeating the procedure (\ref{R1})-(\ref{R8}) we arrive at
\begin{equation}
\sqrt{\overline{V}(t)}:=y^{n}(t)+\frac{\sigma}{2}(w(t)-w(t_{n})),\ t_{n}\leq
t\leq t_{n+1}, \label{R9}%
\end{equation}
satisfying%
\begin{equation}
\left\vert \sqrt{V(t)}-\sqrt{\overline{V}(t)}\right\vert =\left\vert
R^{n}(t)\right\vert \leq r\sum_{m=0}^{n}\left(  D_{1}+\frac{D_{2}}%
{\overline{V}(t_{m})}\right)  \left(  t_{m+1}-t_{m}\right)  ,\text{ \ \ }%
t_{n}\leq t\leq t_{n+1}, \label{R10}%
\end{equation}
with%
\begin{align}
R^{n}(t)  &  :=Y(t;t_{n})-y^{n}(t),\text{ \ \ }t_{n}\leq t\leq t_{n+1}%
,\text{\ and in particular}\label{R11}\\
\left\vert \sqrt{V(t_{n+1})}-\sqrt{\overline{V}(t_{n+1})}\right\vert  &  \leq
r\sum_{m=0}^{n}\left(  D_{1}+\frac{D_{2}}{\overline{V}(t_{m})}\right)  \left(
t_{m+1}-t_{m}\right)  .\nonumber
\end{align}

\begin{remark}
In principle it is possible to use the distribution function $\mathcal{Q}%
$\ (see (\ref{frakQ})) for constructing $\sqrt{\overline{V}(t)}$ for
$t_{n}<t<t_{n+1}.$ However, we rather consider for $t_{n}\leq t\leq t_{n+1}$
the approximation%
\[
\sqrt{\widetilde{V}(t)}:=y^{n}(t)+\frac{\sigma}{2}\widetilde{w}_{n}%
(t),\ t_{n}\leq t\leq t_{n+1},
\]
where (a) for $t_{n+1}<t_{0}+T,$ $\widetilde{w}$ is an arbitrary continuous
function satisfying%
\[
\widetilde{w}(t_{n})=0,\text{ \ }\widetilde{w}(t_{n+1})=w(t_{n+1}%
)-w(t_{n})=r\xi_{n},\text{ \ \ }\max_{t_{n}\leq t\leq t_{n+1}}\text{\ }%
\left\vert \widetilde{w}_{n}(t)\right\vert \leq r,
\]
and (b) for $t_{n+1}=t_{0}+T,$ one may take $\widetilde{w}(t)\equiv0.$ As a
result we get similar to (\ref{mod}) an insignificant increase of the error,%
\[
\left\vert \sqrt{V(t)}-\sqrt{\widetilde{V}(t)}\right\vert \leq r\sum_{m=0}%
^{n}\left(  D_{1}+\frac{D_{2}}{\overline{V}(t_{m})}\right)  \left(
t_{m+1}-t_{m}\right)  +\sigma r,\ t_{n}<t<t_{n+1}.
\]

\end{remark}

Let us consolidate the above procedure in a concise way.

\subsection{Simulation algorithm}

\label{simal}

\begin{itemize}
\item \textit{Set }$\sqrt{\overline{V}(t_{0})}=\sqrt{V(t_{0})}.$\newline

\item \textit{Let the point }$(t_{n},\sqrt{\overline{V}(t_{n})})$\textit{ be
known for an }$n\geq0.$\textit{ Simulate independent random variables }%
$\xi_{n}$\textit{ with }$P(\xi_{n}=\pm1)=1/2,\ $and $\theta_{n}$ as described
in subsection \ref{sim}.\textit{ If }$t_{n}+\theta_{n}<t_{0}+T,$ set
$t_{n+1}=t_{n}+\theta_{n},$ else set $t_{n+1}=t_{0}+T.$\newline

\item \textit{Solve equation (\ref{R09}) on the interval }$[t_{n},t_{n+1}]$
with solution $y^{n}$ and set%
\[
\sqrt{\overline{V}(t_{n+1})}=y^{n}(t_{n+1})+\frac{\sigma}{2}\cdot\left\{
\begin{tabular}
[c]{l}%
$r\xi_{n}$ \ \ if $t_{n+1}<t_{0}+T,$\\
$0$ \ \ \ if $t_{n+1}=t_{0}+T.$%
\end{tabular}
\ \ \ \ \ \right.
\]

\end{itemize}

\bigskip

So, under the assumption (\ref{CT1}) we obtain the estimate (\ref{R10})
(possibly enlarged with a term $\sigma r$). The next theorem shows that if a
trajectory of $V(t)$ under consideration is positive on $[t_{0},t_{0}+T],$
then the algorithm is convergent on this trajectory. We recall that in the
case $2k\lambda\geq\sigma^{2}$ almost all trajectories are positive, hence in
this case {the proposed method is almost surely convergent.}

\medskip

\begin{theorem}
\label{Theorem 8} {Let }$4k\lambda\geq\sigma^{2}${ (i.e., }$\alpha\geq0${).
Then for any positive trajectory }$V(t)>0$\textit{ on }$[t_{0},t_{0}+T]${ the
proposed method is convergent on this trajectory. In particular, there exist
}$\eta>0$ depending on the trajectory $V(\cdot)$ only, and $r_{0}>0$ depending
on $\eta$ {such that }%
\[
\sqrt{\overline{V}(t_{m})}\geq\eta\geq r\sigma,\text{ \ \ for }m=0,1,2,...
\]
for any $r<r_{0}.$ So in particular {(\ref{CT1}) is fulfilled for }all
$m=0,1,...,$ and{ the estimate (\ref{R10}) implies that for any }$r<r_{0},$%
\[
\left\vert \sqrt{V(t_{n+1})}-\sqrt{\overline{V}(t_{n+1})}\right\vert \leq
r\left(  D_{1}+\frac{D_{2}}{\eta^{2}}\right)  T,\text{ \ \ }n=0,1,2,...,\nu.
\]

\end{theorem}

\begin{proof}
Let us define%
\begin{align}
\eta &  :=\frac{1}{2}\min_{t_{0}\leq t\leq t_{0}+T}\sqrt{V(t)}\text{
\ \ and}\nonumber\\
r_{0}  &  :=\min\left(  \frac{\eta}{\sigma},\frac{\eta}{\left(  D_{1}%
+\frac{D_{2}}{\eta^{2}}\right)  T}\right)  , \label{sec}%
\end{align}
and let $r<r_{0}.$ We then claim that for all $m,$
\begin{equation}
\sqrt{\overline{V}(t_{m})}\geq\eta\geq r\sigma. \label{ind}%
\end{equation}
For $m=0$ we trivially have
\[
\sqrt{\overline{V}(t_{0})}=\sqrt{V(t_{0})}\geq2\eta\geq\eta\geq r_{0}%
\sigma\geq r\sigma.
\]
Now suppose by induction that $\sqrt{\overline{V}(t_{j})}\geq\eta$ for
$j=0,...,m.$Then due to (\ref{R11}) we have%
\[
\left\vert \sqrt{V(t_{m+1})}-\sqrt{\overline{V}(t_{m+1})}\right\vert \leq
r\left(  D_{1}+\frac{D_{2}}{\eta^{2}}\right)  T\leq r_{0}\left(  D_{1}%
+\frac{D_{2}}{\eta^{2}}\right)  T\leq\eta
\]
because of (\ref{sec}). Thus, since $\sqrt{V(t_{m+1})}\geq2\eta,$ it follows
that $\sqrt{\overline{V}(t_{m+1})}\geq\eta\geq r\sigma.$ This proves
(\ref{ind}) and the convergence for $r\downarrow0.$
\end{proof}

\begin{remark}
\label{Remark 9}. In the case where $4k\lambda\geq\sigma^{2}>2k\lambda$
trajectories will reach zero with positive probability, that is convergence on
such trajectories is not guaranteed by Theorem~\ref{Theorem 8}. So it is
important to develop some method for continuing the simulations in cases of
very small $\overline{V}(t_{m}).$ One can propose different procedures, for
instance, one can proceed with standard SDE approximation methods relying on
some known scheme suitable for small $V$ (e.g. see \cite{Anders}). However,
the uniformity of the simulation would be destroyed in this way. We therefore
propose in the next section a uniform simulation method that may be started in
a value $\overline{V}(t_{m})$ close to zero.
\end{remark}

\section{Simulation of trajectories near to zero}

Henceforth we assume that $\alpha>0.$ Let us suppose that $\sqrt{\overline
{V}(t_{n})}=y_{n}\geq\sigma r$ and consider conditions that guarantee that
$\sqrt{\overline{V}(t_{n+1})}\geq\sigma r$ under $t_{n+1}\leq t_{0}+T.$ Of
course in the case $\xi_{n}=1$ this is trivially fulfilled, and we thus
consider the case $\xi_{n}=-1,$ yielding%
\[
\sqrt{\overline{V}(t_{n+1})}=y^{n}(t_{n+1})-\frac{\sigma r}{2}=[y_{n}%
^{2}e^{-k(t_{n+1}-t_{n})}+\frac{2\alpha}{k}(1-e^{-k(t_{n+1}-t_{n})}%
)]^{1/2}-\frac{\sigma r}{2}.
\]
We so need%
\begin{align}
y_{n}^{2}e^{-k(t_{n+1}-t_{n})}+\frac{2\alpha}{k}(1-e^{-k(t_{n+1}-t_{n})})  &
\geq\frac{9\sigma^{2}r^{2}}{4},\text{ \ \ i.e.}\nonumber\\
\left(  y_{n}^{2}-\frac{2\alpha}{k}\right)  e^{-k(t_{n+1}-t_{n})}  &
\geq\frac{9\sigma^{2}r^{2}}{4}-\frac{2\alpha}{k}. \label{cp1}%
\end{align}
Since we are interested in properties of algorithms when $r\downarrow0,$ we
may further assume w.l.o.g. that $9\sigma^{2}r^{2}/4-2\alpha/k<0,$ i.e.
\begin{equation}
r<\frac{2}{3}\sqrt{\frac{2\alpha}{k\sigma^{2}}}. \label{Ar}%
\end{equation}
Under assumption (\ref{Ar}), (\ref{cp1}) is obviously fulfilled when
$y_{n}\geq\sqrt{2\alpha/k}.$ If $y_{n}<\sqrt{2\alpha/k}$ we need%
\[
e^{-k(t_{n+1}-t_{n})}\leq\frac{\frac{9\sigma^{2}r^{2}}{4}-\frac{2\alpha}{k}%
}{y_{n}^{2}-\frac{2\alpha}{k}},\text{ \ \ \ hence \ }t_{n+1}-t_{n}\geq\frac
{1}{k}\ln\frac{\frac{2\alpha}{k}-y_{n}^{2}}{\frac{2\alpha}{k}-\frac
{9\sigma^{2}r^{2}}{4}},
\]
which is fulfilled if%
\begin{equation}
y_{n}\geq\frac{3}{2}\sigma r. \label{cp11}%
\end{equation}
Note that (\ref{Ar}) is equivalent with $3\sigma r/2<$ $\sqrt{2\alpha/k},$ and
so (\ref{cp11}) is the condition we were looking for. Conversely, if $\sigma
r\leq y_{n}<3\sigma r/2,$ then $\sqrt{\overline{V}(t_{n+1})}<$ $\sigma r$ with
positive probability. In view of the above considerations, one may carry out
the algorithm of Subsection~\ref{simal} as long as (\ref{cp11}) is fulfilled.
Let us say that $\mathfrak{n}$ was the last step where (\ref{cp11}) was true.
Then the aggregated error of $\sqrt{\overline{V}(t_{n+1})}$ due to the
algorithm up to step $\mathfrak{n}$ may be estimated by (cf. (\ref{CT1}) and
(\ref{CT02})),%
\begin{gather}
\left\vert \sqrt{V(t_{\mathfrak{n}+1})}-\sqrt{\overline{V}(t_{\mathfrak{n}%
+1})}\right\vert \leq\nonumber\\
r\sum_{m=0}^{\mathfrak{n}}\left(  D_{1}+\frac{D_{2}}{y_{m}^{2}}\right)
\left(  t_{m+1}-t_{m}\right)  =r\sum_{m=0}^{\mathfrak{n}}\left(  D_{1}%
+\frac{D_{2}}{\overline{V}(t_{m})}\right)  \left(  t_{m+1}-t_{m}\right)  .
\label{epsn}%
\end{gather}

\bigskip Let us recall that our primal goal is a scheme where $\sqrt
{V(t)}-\sqrt{\overline{V}(t)}\downarrow0,$ almost surely and uniformly in
$t_{0}\leq t\leq t_{0}+T.$ In this respect, and in particular in the case
$\sigma^{2}>2k\lambda$ where trajectories may attain zero with positive
probability, it is not recommended to carry out scheme \ref{simal} all the way
through until (\ref{cp11}) is not satisfied anymore. Indeed, if the trajectory
attains zero, the worst case almost sure error bound would then be when all
$y_{m}$ would be close to $3\sigma r/2,$ hence of order $O(1/r).$ That is, no
convergence on such trajectories. We therefore propose to perform scheme
\ref{simal} up to a (stopping) index $\mathfrak{m},$ defined by%
\begin{equation}
\sqrt{\overline{V}(t_{k})}\geq\dfrac{1}{2}Ar^{a}>\frac{3}{2}\sigma
r,\ k=0,1,...,\mathfrak{m},\text{ and }\sqrt{\overline{V}(t_{\mathfrak{m}+1}%
)}<\dfrac{1}{2}Ar^{a}, \label{CZ1}%
\end{equation}
where $A$ is a positive constant and $0<a<1/2$ is to be determined suitably. A
pragmatic choice would be $a=1/3$ (see Remark~\ref{ma}). Due to (\ref{CZ1})
and (\ref{epsn}) with $\mathfrak{n}$\ replaced by $\mathfrak{m}$ we then have,%
\begin{equation}
\left\vert \sqrt{V(t_{\mathfrak{m}+1})}-\sqrt{\overline{V}(t_{\mathfrak{m}%
+1})}\right\vert \leq r\sum_{m=0}^{\mathfrak{m}}\left(  D_{1}+\frac{4D_{2}%
}{A^{2}r^{2a}}\right)  \left(  t_{m+1}-t_{m}\right)  \leq:D_{3}r^{1-2a}.
\label{CZ2}%
\end{equation}
for some constant $D_{3}>0.$

Let us now fix a realization $t_{m+1}:=t_{\mathfrak{m}+1},$ and consider two
solutions of equation (\ref{In1}) starting at the moment $t_{m+1}$ from
$\overline{v}:=\overline{V}(t_{m+1})$ (known value) and $v=V(t_{m+1})$ (true
but unknown value), denoted by $V_{t_{m+1},\overline{v}}$ and $V_{t_{m+1},v},$
respectively, Let $\vartheta_{x},$ $0\leq x<A^{2}r^{2a},$ be the first time at
which the solution $V_{t_{m+1},x}(t+t_{m+1})$ of (\ref{In1}) attains the level
$A^{2}r^{2a},$ hence%
\[
0\leq V_{t_{m+1},x}(t_{m+1}+t)<A^{2}r^{2a},\ 0\leq t<\vartheta_{x}%
,\ V_{t_{m+1},x}(t_{m+1}+\vartheta_{x})=A^{2}r^{2a}.
\]
A construction of the distribution function of $\vartheta_{x}$ is worked out
in Section \ref{tetax}. Let us now denote $t_{m+2}:=t_{m+1}+\vartheta$ with
$\vartheta=\vartheta_{\overline{v}}.$ (For simplicity and w.l.og. we assume
that $t_{m+2}<t_{0}+T).$ We then naturally set
\[
\overline{V}(t_{m+2})=V_{t_{m+1},\overline{v}}(t_{m+1}+\vartheta)=A^{2}%
r^{2a}.
\]
The solutions $V_{t_{m+1},\overline{v}}$ and $V_{t_{m+1},v}$ correspond to two
solutions $\overline{Y}(t,t_{m+1})$ and $Y(t,t_{m+1})$ of (\ref{OS4}) with
$\varphi(t)=w(t)-w(t_{m+1}),$ $t\geq t_{m+1},$ starting in $\overline
{Y}(t_{m+1},t_{m+1})=\sqrt{\overline{v}}$ and $Y(t_{m+1},t_{m+1})=\sqrt{v},$
respectively. Due to Lemma \ref{Lemma 2}, see Remark \ref{L2*}, and
(\ref{CZ2}) it thus follows that%
\begin{gather}
\left\vert \sqrt{V_{t_{m+1},v}(t)}-\sqrt{V_{t_{m+1},\overline{v}}%
(t)}\right\vert =\left\vert Y(t,t_{m+1})-\overline{Y}(t,t_{m+1})\right\vert
\leq\left\vert \sqrt{V(t_{m+1})}-\sqrt{\overline{V}(t_{m+1})}\right\vert
\nonumber\\
\leq D_{3}r^{1-2a},\text{ \ }t_{m+1}\leq t\leq t_{m+2}, \label{ba1}%
\end{gather}
and in particular%
\[
\sqrt{V(t_{m+2})}-\sqrt{\overline{V}(t_{m+2})}\leq D_{3}r^{1-2a}.
\]
In contrast to the previous steps we now specify the behavior of $\overline
{V}(t)$ on $[t_{m+1},t_{m+2}]$ by
\begin{equation}
\overline{V}(t)=V_{t_{m+1},\overline{v}}(t),\text{ \ }t_{m+1}\leq t\leq
t_{m+2}, \label{balke}%
\end{equation}
which we actually do not know. However, we do know that $\overline{V}%
(t_{m+2})=V_{t_{m+1},\overline{v}}(t_{m+2})=$ $A^{2}r^{2a},$ and that
$\overline{V}$ is bounded on $[t_{m+1},t_{m+2}]$ by $A^{2}r^{2a}.$ Therefore,
if we just take a straight line $L(t)$\ that connects the points
$(t_{m+1},\sqrt{\overline{v}})$ and $(t_{m+2},Ar^{a})$ as an approximation for
$\sqrt{\overline{V}(t)},$ then $\sqrt{\overline{V}(t)}-L(t)\leq Ar^{a},$
$\ t_{m+1}\leq t\leq t_{m+2}.$ By (\ref{ba1}) and (\ref{balke}) we then also
have
\[
\sqrt{V(t)}-L(t)\leq Ar^{a},\ t_{m+1}\leq t\leq t_{m+2}.
\]
Thus, the accuracy of the approximation to $\sqrt{V}$ for $0\leq t\leq
t_{m+1}$ outside the band $\left(  0,\dfrac{1}{2}Ar^{a}\right)  $ is of order
$O(r^{1-2a}),$ and for $t_{m+1}<t<t_{m+2}$ inside the band $\left(
0,Ar^{a}\right)  $ of order $O(r^{a}).$ But, at the boundary point
$\overline{V}(t_{m+2})=$ $A^{2}r^{2a}$ the accuracy is of order $O(r^{1-2a})$
again. Finally, the scheme may be continued from the state
\[
\sqrt{\overline{V}(t_{m+2})}=Ar^{a}%
\]
with the algorithm of Subsection \ref{simal}.

\begin{remark}
\label{ma} From the above construction it is clear that for $a=1/3$ in
(\ref{CZ1}) the accuracy for $0\leq t\leq t_{m+1}$ outside the band $\left(
0,\dfrac{1}{2}Ar^{a}\right)  ,$ and for $t_{m+1}<t<t_{m+2}$ inside the band
$\left(  0,Ar^{a}\right)  $ are of the same order. However, an exponent
$0<a<1/3$ would give a higher accuracy outside the band $\left(  0,\dfrac
{1}{2}Ar^{a}\right)  $ and at the exit points of the band $\left(
0,Ar^{a}\right)  ,$ while inside the band the accuracy is worse but uniformly
bounded by $Ar^{a}.$
\end{remark}

\subsection{Simulation of $\vartheta_{x}$}

\label{tetax} In order to carry out the above simulation method for
trajectories near zero we have to find the distribution function of
$\vartheta_{x}=\vartheta_{x,l},$ where $\vartheta_{x,l}$ is the first-passage
time of the trajectory $X_{0,x}(s),$ to the level $l.$ For this it is more
convenient to change notation and to write (\ref{In1}) in the form%
\begin{equation}
dX(s)=k(\lambda-X(s))ds+\sigma\sqrt{X}dw(s),\ X(0)=x, \label{Z1}%
\end{equation}
where without loss of generality we take the initial time to be $s=0.$ The
function%
\[
u(t,x):=P(\vartheta_{x,l}<t),
\]
is the solution of the first boundary value problem of parabolic type
(\cite{MT1}, Ch. 5, Sect. 3)%
\begin{equation}
\frac{\partial u}{\partial t}=\frac{1}{2}\sigma^{2}x\frac{\partial^{2}%
u}{\partial x^{2}}+k(\lambda-x)\frac{\partial u}{\partial x},\ t>0,\ 0<x<l,
\label{Z2}%
\end{equation}
with initial data%
\begin{equation}
u(0,x)=0, \label{Z3}%
\end{equation}
and boundary conditions%
\begin{equation}
u(t,0)\text{ is bounded, }u(t,l)=1. \label{Z4}%
\end{equation}
To get homogeneous boundary conditions we introduce $v=u-1.$ The function $v$
then satisfies:%
\begin{equation}
\frac{\partial v}{\partial t}=\frac{1}{2}\sigma^{2}x\frac{\partial^{2}%
v}{\partial x^{2}}+k(\lambda-x)\frac{\partial v}{\partial x},\ t>0,\ 0<x<l,
\label{Z5}%
\end{equation}%
\begin{equation}
v(0,x)=-1;\ v(t,0)\text{ is bounded, }v(t,l)=1. \label{Z6}%
\end{equation}
The problem (\ref{Z5})-(\ref{Z6}) can be solved by the method of separation of
variables. In this way the Sturm-Liouville problem for the confluent
hypergeometric equation (the Kummer equation) arises . This problem is rather
complicated however. Below we are going to solve an easier problem as a good
approximation to (\ref{Z5})-(\ref{Z6}). Along with (\ref{Z1}), let us consider
the equations%
\begin{align}
dX^{+}(s)  &  =k\lambda ds+\sigma\sqrt{X^{+}}dw(s),\ X^{+}(0)=x,\label{Z7}\\
dX^{-}(s)  &  =k(\lambda-l)ds+\sigma\sqrt{X^{-}}dw(s),\ X^{-}(0)=x, \label{Z8}%
\end{align}
with $0\leq l<\lambda.$ It is not difficult to prove the following
inequalities%
\begin{equation}
X^{-}(s)\leq X(s)\leq X^{+}(s). \label{Z9}%
\end{equation}
According to (\ref{Z9}), we consider three boundary value problems: first
(\ref{Z2})-(\ref{Z4}) and next similar ones for the equations%
\begin{align}
\frac{\partial u^{+}}{\partial t}  &  =\frac{1}{2}\sigma^{2}x\frac
{\partial^{2}u^{+}}{\partial x^{2}}+k\lambda\frac{\partial u^{+}}{\partial
x},\ t>0,\ 0<x<l,\nonumber\\
\frac{\partial u^{-}}{\partial t}  &  =\frac{1}{2}\sigma^{2}x\frac
{\partial^{2}u^{-}}{\partial x^{2}}+k(\lambda-l)\frac{\partial u^{-}}{\partial
x},\ t>0,\ 0<x<l. \label{min}%
\end{align}
From (\ref{Z9}) it follows that%
\[
u^{-}(t,x)\leq u(t,x)\leq u^{+}(t,x),
\]
hence%
\[
v^{-}(t,x)\leq v(t,x)\leq v^{+}(t,x),
\]
where $v^{-}=u^{-}-1,\ v^{+}=u^{+}-1.$

As the band $0<x<l=A^{2}r^{2a},$ for a certain $a>0,$ is narrow due to small
enough $r,$ the difference $v^{+}-v^{-}$ will be small and so we can consider
the following problem%
\begin{equation}
\frac{\partial v^{+}}{\partial t}=\frac{1}{2}\sigma^{2}x\frac{\partial
^{2}v^{+}}{\partial x^{2}}+k\lambda\frac{\partial v^{+}}{\partial
x},\ t>0,\ 0<x<l,\label{Z10}%
\end{equation}%
\begin{equation}
v^{+}(0,x)=-1;\ v^{+}(t,0)\text{ is bounded, }v^{+}(t,l)=0,\label{Z11}%
\end{equation}
as a good approximation of (\ref{Z5})-(\ref{Z6}). Henceforth we write
$v:=v^{+}.$ By separation of variables we get as elementary independent
solutions to (\ref{Z10}), $\mathcal{T}(t)\mathcal{X}(x),$ where%
\begin{gather}
\mathcal{T}^{\prime}(t)+\mu\mathcal{T}(t)=0,\text{ \ \ i.e. \ \ }%
\mathcal{T}(t)=\mathcal{T}_{0}e^{-\mu t},\text{ \ \ }\mu>0,\text{
\ \ and}\label{chi0}\\
\frac{1}{2}\sigma^{2}x\mathcal{X}^{\prime\prime}+k\lambda\mathcal{X}^{\prime
}+\mu\mathcal{X}=0,\text{ \ \ }\mathcal{X}(0+)\text{ is bounded,
\ }\mathcal{X}(l)=0.\text{ }\label{Z13}%
\end{gather}
It can be verified straightforwardly that the solution of (\ref{Z13}) can be
obtained in terms of Bessel functions of the first kind (e.g. see \cite{BE}),
\[
\mathcal{X}(x)=\mathcal{X}_{\gamma}^{\pm}(x):=x^{\gamma}J_{\pm2\gamma}\left(
\sigma^{-1}\sqrt{8\mu x}\right)  =x^{\gamma}O(x^{\pm\gamma})\ \ \ \ \text{if}%
\ \text{\ \ }\ x\downarrow0,
\]
with
\begin{equation}
\gamma:=\frac{1}{2}-\frac{k\lambda}{\sigma^{2}}.\label{gam}%
\end{equation}
Since $\mathcal{X}(x)$ has to be bounded for $x\downarrow0$ we may take
(regardless the sign of $\gamma$ (!))
\begin{equation}
\mathcal{X}(x)=\mathcal{X}_{\gamma}^{-}(x)=:\mathcal{X}_{\gamma}(x)=x^{\gamma
}J_{-2\gamma}\left(  \sigma^{-1}\sqrt{8\mu x}\right)  .\label{chi}%
\end{equation}
In our setting we have $\alpha>0,$ i.e. $\gamma<1/4.$

The following derivation of a Fourier-Bessel series for $v$ is standard but
included for  convenience of the reader. Denote the positive zeros of
$J_{\nu}$ by $\pi_{\nu,m},$ for example,
\begin{equation}
J_{1/2}(x)=\sqrt{\dfrac{2}{\pi x}}\sin x\text{ \ \ and \ \ }\pi_{1/2,m}%
=m\pi,\text{ \ \ }m=1,2,...\label{jex}%
\end{equation}
Then the (homogeneous) boundary condition $\mathcal{X}_{\gamma}(l)=0$ yields%
\begin{equation}
\sigma^{-1}\sqrt{8\mu l}=\pi_{-2\gamma,m},\text{ \ \ i.e., \ \ }\mu_{m}%
:=\frac{\sigma^{2}\pi_{-2\gamma,m}^{2}}{8l}\label{chi1}%
\end{equation}
and we have%
\[
\mathcal{X}_{\gamma,m}(x):=x^{\gamma}J_{-2\gamma}\left(  \sigma^{-1}\sqrt
{8\mu_{m}x}\right)  =x^{\gamma}J_{-2\gamma}\left(  \pi_{-2\gamma,m}\sqrt
{\frac{x}{l}}\right)  .
\]
By the well-known orthogonality relation%
\[
\int_{0}^{1}zJ_{-2\gamma}(\pi_{-2\gamma,k}z)J_{-2\gamma}(\pi_{-2\gamma
,k^{\prime}}z)dz=\frac{\delta_{k,k^{\prime}}}{2}J_{-2\gamma+1}^{2}%
(\pi_{-2\gamma,k}),
\]
we get by setting $z=\sqrt{x/l}$
\begin{align*}
\int_{0}^{l}J_{-2\gamma}(\pi_{-2\gamma,m}\sqrt{\frac{x}{l}})J_{-2\gamma}%
(\pi_{-2\gamma,m^{\prime}}\sqrt{\frac{x}{l}})dx &  =l\delta_{m,m^{\prime}%
}J_{-2\gamma+1}^{2}(\pi_{-2\gamma,m}),\text{ \ \ hence}\\
\int_{0}^{l}\mathcal{X}_{\gamma,m}(x)\mathcal{X}_{\gamma,m^{\prime}%
}(x)x^{-2\gamma}dx &  =l\delta_{m,m^{\prime}}J_{-2\gamma+1}^{2}(\pi
_{-2\gamma,m}).
\end{align*}
Now set%
\begin{equation}
v(t,x)=\sum_{m=1}^{\infty}\beta_{m}e^{-\mu_{m}t}\mathcal{X}_{\gamma
,m}(x),\text{ \ \ \ }0\leq x\leq l.\label{ser}%
\end{equation}
For $t=0$ we have due to the initial condition $v(0,x)=-1,$%
\[
-1=\sum_{m=1}^{\infty}\beta_{m}\mathcal{X}_{\gamma,m}(x).
\]
So for any $p=1,2,...,$%
\begin{align}
-\int_{0}^{l}\mathcal{X}_{\gamma,p}(x)x^{-2\gamma}dx &  =\beta_{p}%
lJ_{-2\gamma+1}^{2}(\pi_{-2\gamma_{\kappa},p}),\text{ \ \ i.e.}\nonumber\\
\beta_{p} &  =-\frac{\int_{0}^{l}\mathcal{X}_{\gamma,p}(x)x^{-2\gamma}%
dx}{lJ_{-2\gamma+1}^{2}(\pi_{-2\gamma,p})}.\label{chi2}%
\end{align}
Further it holds that%
\begin{align*}
\int_{0}^{l}\mathcal{X}_{\gamma,p}(x)x^{-2\gamma}dx &  =\int_{0}^{l}%
x^{-\gamma}J_{-2\gamma}\left(  \pi_{-2\gamma,p}\sqrt{\frac{x}{l}}\right)  dx\\
&  =2l^{-\gamma+1}\int_{0}^{1}z^{-2\gamma+1}J_{-2\gamma}\left(  \pi
_{-2\gamma,p}z\right)  dz\\
&  =2l^{-\gamma+1}\frac{J_{-2\gamma+1}\left(  \pi_{-2\gamma,p}\right)  }%
{\pi_{-2\gamma,p}}%
\end{align*}
by well-known identities for Bessel functions (e.g. see \cite{BE}), and (\ref{chi2}) thus
becomes
\begin{equation}
\beta_{p}=-\frac{2}{l^{\gamma}\pi_{-2\gamma,p}J_{-2\gamma+1}(\pi_{-2\gamma
,p})},\text{ \ \ }p=1,2,.....\label{chi3}%
\end{equation}
So, from $v=u-1,$ (\ref{chi0}) (\ref{chi}), (\ref{chi1}), (\ref{chi3}), and
(\ref{ser}) we finally obtain%
\begin{equation}
u(t,x)=1-2x^{\gamma}l^{-\gamma}\sum_{m=1}^{\infty}\frac{J_{-2\gamma}\left(
\pi_{-2\gamma,m}\sqrt{\frac{x}{l}}\right)  }{\pi_{-2\gamma,m}J_{-2\gamma
+1}(\pi_{-2\gamma,m})}\exp\left[  -\frac{\sigma^{2}\pi_{-2\gamma,m}^{2}}%
{8l}t\right]  \,,\text{ \ \ }0\leq x\leq l.\label{FB1}%
\end{equation}

\begin{example}
\label{test} For $\gamma=-1/4$ we get from (\ref{FB1}) by (\ref{jex})
straightforwardly,%
\[
u(t,x)=1+\frac{2}{\pi}\sqrt{\frac{l}{x}}\sum_{m=1}^{\infty}\frac{(-1)^{m}}%
{m}\sin\left(  \pi m\sqrt{\frac{x}{l}}\right)  \exp\left[  -\frac{\sigma
^{2}\pi^{2}m^{2}}{8l}t\right]  .
\]

\end{example}

\bigskip

For solving (\ref{min}) we set $\lambda^{-}:=\lambda-l,$ and then apply the
Fourier-Bessel series (\ref{FB1}) with $\gamma$ replaced by
\begin{equation}
\gamma^{-}:=\frac{1}{2}-\frac{k\lambda^{-}}{\sigma^{2}}=\gamma+\frac
{kl}{\sigma^{2}}. \label{kap}%
\end{equation}

\begin{example}
We now consider some numerical examples concerning $u^{+}=u$ in (\ref{FB1})
and $u^{-}$ given by (\ref{FB1}) due to (\ref{kap}). Note that actually in
(\ref{FB1}) the function $u$ only depends on $\sigma,l,$ and $\gamma.$ That
is, $u$ depends on $\sigma,l,$ and the product $k\lambda.$ Let us consider a
CIR process with $\sigma=1,$ $\lambda=1,$ $k=0.75,$ and let us take $l=0.1.$
We then compare $u^{+},$ which is given by (\ref{FB1}) for $\gamma=-0.25$ due
to (\ref{gam}) (see Example~\ref{test}), with $u^{-}$ given by (\ref{FB1}) for
$\gamma^{-}=-0.175$ due to (\ref{kap}). The results are depicted in
Figure~\ref{plot}. The sums corresponding to (\ref{FB1}) are computed with
five terms (more terms did not give any improvement).
\end{example}

\subsubsection*{Normalization of $u(t,x)$}

For practical applications it is useful to normalize (\ref{FB1}) in the
following way. Let us treat $\gamma$ as essential but fixed parameter,
introduce as new parameters
\[
\frac{x}{l}=\widetilde{x},\text{ \ \ }0<\widetilde{x}\leq1,\text{ \ \ }%
\frac{\sigma^{2}t}{8l}=\widetilde{t},\text{ \ \ }\widetilde{t}\geq0,
\]
and consider the function%
\[
\widetilde{u}(\widetilde{t},\widetilde{x}):=1-2\widetilde{x}^{\gamma}%
\sum_{m=1}^{\infty}\frac{J_{-2\gamma}\left(  \pi_{-2\gamma,m}\sqrt
{\widetilde{x}}\right)  }{\pi_{-2\gamma,m}J_{-2\gamma+1}(\pi_{-2\gamma,m}%
)}\exp\left[  -\pi_{-2\gamma,m}^{2}\widetilde{t}\right]  \,,\ \ 0<\widetilde
{x}\leq1,\text{ \ \ }\widetilde{t}\geq0,
\]
that is connected to (\ref{FB1}) via%
\[
\widetilde{u}(\widetilde{t},\widetilde{x})=\widetilde{u}(\frac{\sigma^{2}%
t}{8l},\frac{x}{l})=u(\frac{8l\widetilde{t}}{\sigma^{2}},l\widetilde{x}).
\]
For simulation of $\vartheta_{x}$ we need to solve the equation%
\[
u(\vartheta_{x},x)=U,\text{ \ where }U\sim\text{Uniform}[0,1].
\]
For this we set $\widetilde{x}=x/l$ and solve the normalized equation
$\widetilde{u}(\widetilde{\vartheta}_{\widetilde{x}},\widetilde{x})=U,$ and
then take%
\[
\text{ }\vartheta_{x}=\frac{8l}{\sigma^{2}}\widetilde{\vartheta}%
_{\widetilde{x}}.
\]
Note that
\[
P(\vartheta_{x}<t)=P(\widetilde{\vartheta}_{\widetilde{x}}<\frac{\sigma^{2}%
t}{8l})=\widetilde{u}(\frac{\sigma^{2}t}{8l},\frac{x}{l}).
\]
We have plotted in Figure~\ref{normal} the normalized function $\widetilde
{u}(\widetilde{t},\widetilde{x})$ for $\gamma=-1/4.$

\begin{figure}[ptb]
\begin{center}
\includegraphics[scale=0.75]{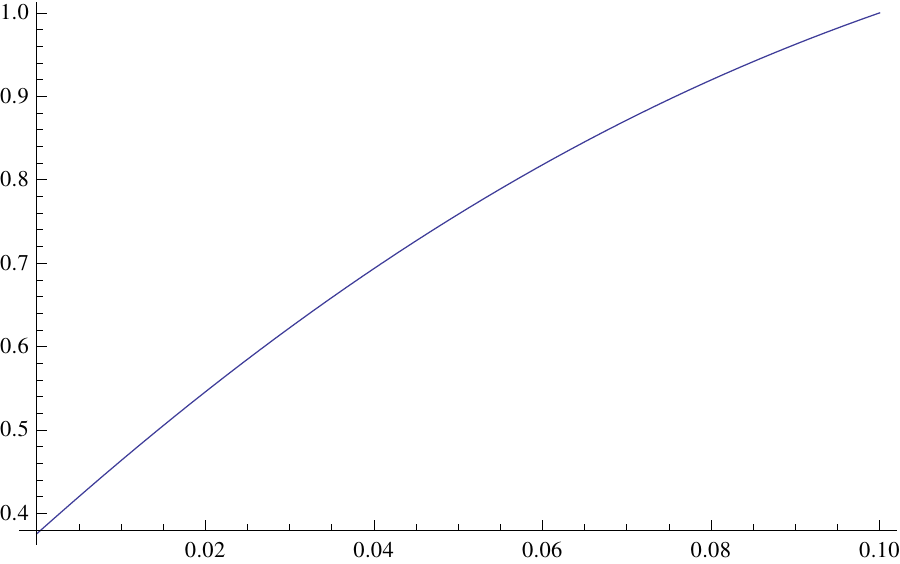} \includegraphics[scale=0.75]{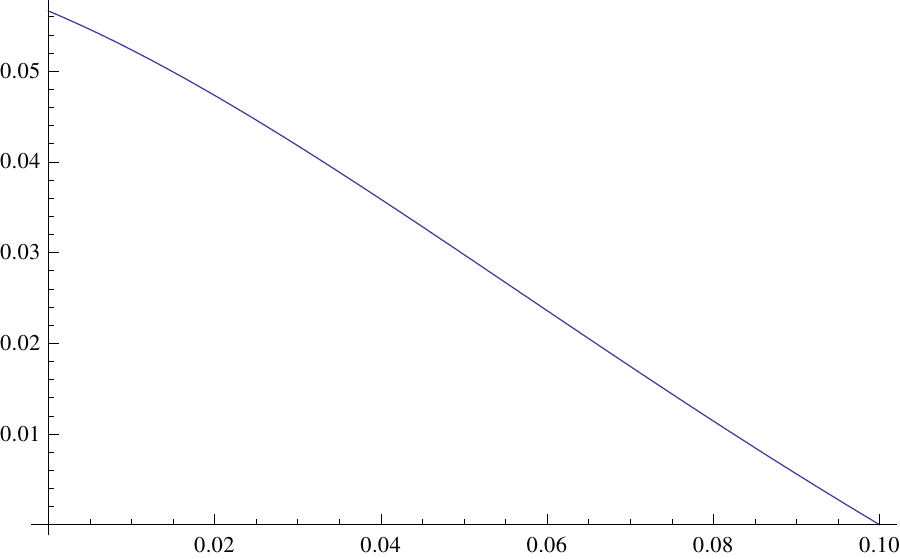}
\end{center}
\caption{Upper panel $u^{+}(0.1,x),$ lower panel $u^{+}(0.1,x)-u^{-}(0.1,x),$
for $0\le x\le0.1$}%
\label{plot}%
\end{figure}

\begin{figure}[ptb]
\begin{center}
\includegraphics[scale=0.75]{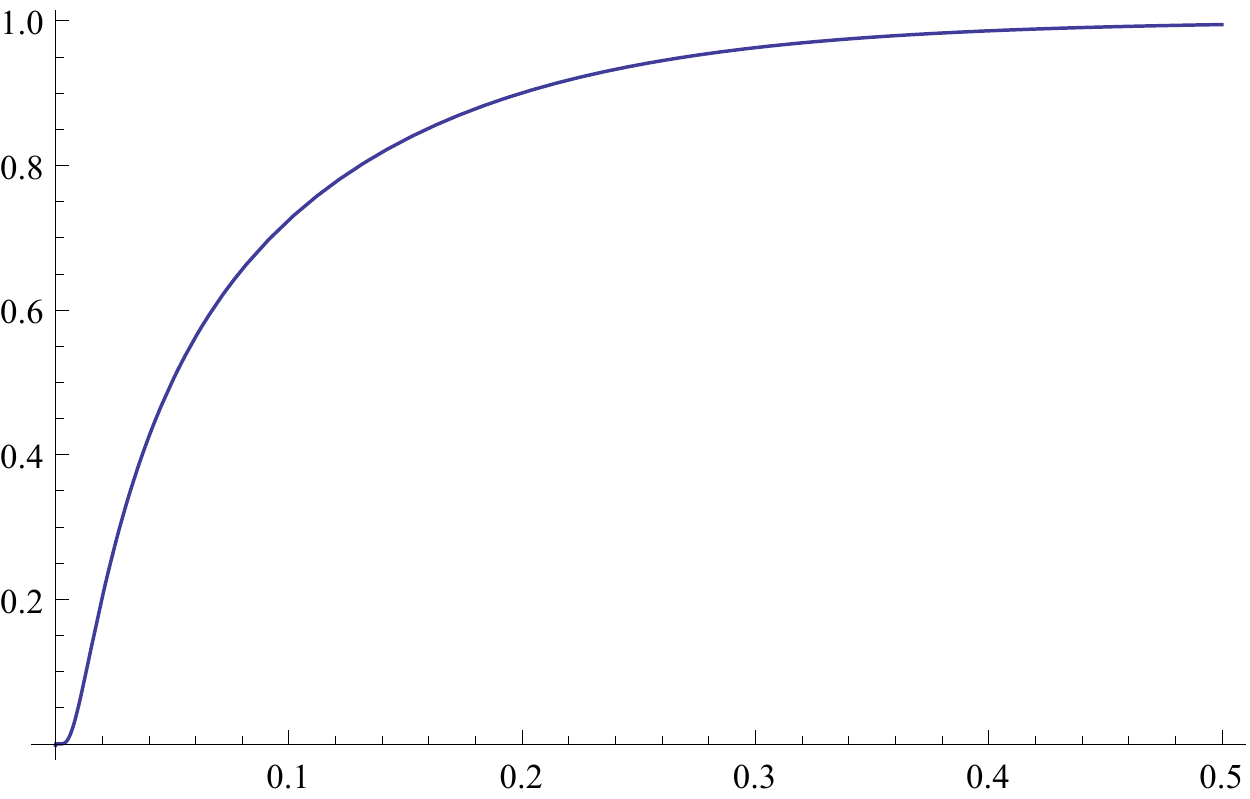}
\end{center}
\caption{Normalized distribution function $\widetilde u(\widetilde
t,\widetilde x)$ for $\gamma=-1/4$}%
\label{normal}%
\end{figure}

\end{document}